\documentclass[reqno,12pt]{amsart}
\usepackage[utf8]{inputenc}
\usepackage{amssymb}
\usepackage{amsmath, amsthm, amsfonts}
\usepackage{color}

\usepackage[a4paper]{geometry}

\usepackage[colorlinks=true,linkcolor=blue,citecolor=blue,urlcolor=blue,breaklinks]{hyperref}
\hypersetup{pdftitle={The Fokker-Planck equation for bosons in 2D: well-possedness and asymptotic analysis}}
\hypersetup{pdfauthor={José A. Cañizo, José A. Carrillo, Philippe Laurençot and Jesús Rosado}}

\usepackage[all]{xy}

\usepackage{breakurl}

\newtheorem{thm}{Theorem}[section]

\newtheorem{lem}[thm]{Lemma}
\newtheorem{prp}[thm]{Proposition}

\theoremstyle{definition}
\newtheorem{dfn}[thm]{Definition}
\theoremstyle{remark}
\newtheorem{rem}[thm]{Remark}
\theoremstyle{example}

\newcommand{\be}{\begin{equation}}
\newcommand{\ee}{\end{equation}}
\newcommand{\ba}{\begin{eqnarray}}
\newcommand{\ea}{\end{eqnarray}}

\newcommand{\dv}{\mathrm{div}}

\newcommand{\p}{\partial}

\def\N{\mathbb{N}}

\def\R{\mathbb{R}}

\def\a{\alpha}

\def\d{\delta}

\def\calF{\mathcal{F}}

\def\rmd{\mathrm{d}}
\def\rme{\mathrm{e}}

\def \d{\,\mathrm{d}}
\def \ddt{\frac{\mathrm{d}}{\mathrm{d}t}}

\def\ird{\int_{\R^d}}

\DeclareMathOperator{\Lip}{Lip}

\title[The Fokker-Planck equation for bosons in 2D]{The Fokker-Planck equation for bosons in 2D: {well-posedness} and asymptotic behavior}

\author{Jos\'e A. Ca\~nizo}
\address{\textbf{Jos\'{e} A.  Ca\~{n}izo}, Departamento de Matem\'atica
  Aplicada, Universidad de Granada, 18071 Granada, Spain}
\email{\href{mailto:canizo@ugr.es}{canizo@ugr.es}}

\author{Jos\'e A. Carrillo}
\address{\textbf{Jos\'{e} A.  Carrillo},
  Department of Mathematics, Imperial College London, South Kensington
  Campus, London SW7 2AZ, UK}
\email{\href{mailto:carrillo@imperial.ac.uk}{carrillo@imperial.ac.uk}}

\author{Philippe Lauren\c{c}ot}
\address{\textbf{Philippe Lauren\c{c}ot}, {Institut de Math\'ematiques de Toulouse, UMR~5219, Universit\'e de Toulouse, CNRS,} F–31062 Toulouse Cedex 9, France}
\email{\href{mailto:laurenco@math.univ-toulouse.fr}{laurenco@math.univ-toulouse.fr}}

\author{Jes\'us Rosado}
\address{\textbf{Jes\'{u}s Rosado}, Departamento de Matem\'atica, Universidad de Buenos Aires}
\email{\href{mailto:jrosado@dm.uba.ar}{jrosado@dm.uab.ar}}

\begin{document}

\begin{abstract}
  We show that solutions of the 2D Fokker-Planck equation for
  bosons are defined globally in time and converge to equilibrium, and
  this convergence is shown to be exponential for radially symmetric
  solutions. The main observation is that a variant of the Hopf-Cole
  transformation relates the 2D equation in radial coordinates to the
  usual linear Fokker-Planck equation. Hence, radially symmetric
  solutions can be computed analytically, and our results for general
  (non radially symmetric) solutions follow from comparison and
  entropy arguments. In order to show convergence to equilibrium we
  also prove a version of the Csiszár-Kullback inequality for the
  Bose-Einstein-Fokker-Planck entropy functional.
\end{abstract}


\maketitle

\tableofcontents

\section{Introduction}
\label{sec:intro}

Introduced in the early 1900s \cite{planck1917, fokker1914},
Fokker-Planck equations have demonstrated to be a powerful tool to
study macroscopic qualitative properties of systems of many
interacting particles. They are particularly relevant in the field of
statistical mechanics since the probability laws of Ornstein-Uhlenbeck
or Langevin processes follow this type of partial differential
equation \cite{gardiner, Risken, bib:f1,PhysRevA.37.4419}. They are
usually obtained as grazing collision limits of the Boltzmann equation
or expansions from the master equation in classical \cite{goudon,DV}
and in quantum contexts
\cite{PhysRevA.31.3761,PhysRevD.45.2843,lux,bib:k, bib:rk,AT}, and
{find applications} in multiple { areas}
\cite{CPT,CCT,GH}. In this way, nonlinear Fokker-Planck equations
appear naturally as simplified kinetic equations with typically the
same equilibrium states as the more physically grounded Boltzmann-type
kinetic or master equations, but keeping some of the essential features and difficulties.

In this paper, we work on the particular case of the model introduced
in the nineties by Kaniadakis and Quarati \cite{Kaniadakis93} to
describe the distribution of quantum particles in a spatially
homogeneous setting leading to possible concentrations in finite
time. From the physics point of view, it is quite relevant due to its
relation to the Bose-Einstein condensate phenomena
\cite{Escobedo2001471, Escobedo2004208,citeulike:3197638}. Different
models have also been proposed to study this process, mostly based
upon the Schr\"odinger equation and its Gross-Pitaevskii variants
\cite{AH08,ESY10,S01}.

The basic idea is to introduce a nonlinear Fokker-Planck operator such
that the Bose-Einstein distributions are their only possible steady
states. These distributions are the family of functions of the form
\begin{equation}
  \label{eq:stationary}
  f_\infty^\beta(v) := \frac{1}{\beta \exp(|v|^2 / 2) - 1},
  \qquad v \in \R^d,
\end{equation}
where $\beta \geq 1$ is a parameter. They obviously satisfy the
equation
$$
\nabla_v \log\left(\frac{f_\infty^\beta}{1+f_\infty^\beta}\right){(v)} + v =0 \quad \Longleftrightarrow \quad
\nabla_v f_\infty^\beta{(v)} + v f_\infty^\beta(1+f_\infty^\beta){(v)}=0 
$$
for all $\beta> 1$. Clearly, for each $\beta>1$  $f_\infty^\beta$ is a stationary solution to the following initial value problem for $f = f(t,v)$, with
$t \geq 0$ and $v \in \R^d$:
\begin{subequations}
  \label{befp}
  \begin{align}
    \label{befp-pde}
    \p_t f = \Delta f + \dv \big(vf(1+f)\big), &\qquad t > 0,\ v \in
    \R^d,
    \\
    \label{befp-ic}
    f(0,v) = f_0(v), &\qquad v \in \R^d.
  \end{align}
\end{subequations}
The \emph{mass} at time $t$ of a solution $f$ to eq. \eqref{befp} is
defined by its integral, and it is a conserved quantity of the time
evolution, that is,
$$
\int_{\R^d} f(t,v) \d v = m := \int_{\R^d} f_0(v) \d v\ , \qquad t>0\ .
$$
Problem~\eqref{befp} is referred to as the Cauchy problem for the
\emph{Bose-Einstein-Fokker-Planck} (BEFP) equation or the
\emph{Fokker-Planck equation for bosons}. It is easy to check that the
functions with $\beta>1$ given by \eqref{eq:stationary} are the only
possible finite-mass smooth stationary solutions of \eqref{befp}.

In dimension $d \geq 3$ there are no stationary solutions with mass
larger than that of $f_\infty^1$, while in dimensions $d = 1,2$ the
mass of $f_\infty^1$ is infinite and there is exactly one stationary
solution for each fixed finite mass. We define the \emph{critical
  mass} $m_c$ by
\begin{equation}
  \label{eq:critical-mass}
  {m_c} := \int_{\R^d} f_\infty^1(v) \d v
  = \int_{\R^d} \frac{1}{\exp(|v|^2 / 2) - 1} \d v,
\end{equation}
which is finite in dimensions $d \geq 3$ and equal to $\infty$ if
$d=1,2$. The mass of $f_\infty^\beta$ is actually explicitly computable in
dimension $d=2$:
\begin{equation*}
  \int_{\R^2} f_\infty^\beta(v) \d v
  = 2 \pi \log\left( \frac{\beta}{\beta - 1} \right)
  \quad \text{ for $\beta > 1$.}
\end{equation*}
Since it is expected that solutions to \eqref{befp} converge to an
equilibrium with the same mass if available, one
  anticipates that solutions with mass $m \leq m_c$ are
globally defined for all $t \in [0,\infty)$, while it is
  conjectured that solutions to \eqref{befp} with mass 
  $m > m_c$ will blow up at a certain finite time $T^*$.

Despite its simple formulation, very little can be said about the
Cauchy problem for the BEFP equation \eqref{befp} beyond the local
existence of solutions \cite{crs08,clr08,T11}, which can be proved by
usual fixed-point arguments.  It was proved by Toscani \cite{T11}, via
a contradiction argument involving the second moment, that in
dimension $d \geq 3$ there are indeed certain solutions which blow up
if the initial mass is sufficiently large or, for any given
supercritical mass, if the initial second moment is small enough. The
nature of the blow-up and its profile is an open question in
$d\geq 3$. The numerical approach becomes then of even greater
interest in clarifying the long time asymptotics of quantum models
related to Bose-Einstein distributions, see for instance
\cite{MR3296281,Kaniadakis93}. We finally remark that in one dimension
a direct minimization of the entropy functional for related equations
can lead to singular parts in the equilibrium states
\cite{BGT}.

We focus here on the $2$-dimensional case, where the critical mass $m_c$ is infinite and show that smooth solutions exist
globally in time under certain mild conditions on the initial data. We
always consider non-negative solutions to \eqref{befp}. One of the
main difficulties that equation~\eqref{befp} presents when attempting
to show the existence of global solutions is the lack of a priori
estimates that would allow to control the $L^p$-norm of solutions for
{ any} $p > 1$.

Our analysis is based on the following change of variables, where, for the sake of clarity we shall omit, in general, the
  dependence on $t$ and $r$ or $v$ of the different quantities that we
  will introduce. Consider a radially symmetric smooth solution
$f(t,v) = |v|^{1-d} \varphi(t,|v|)$ to equation~\eqref{befp}, with
finite mass $m = \|f_0\|_{L^1} > 0$. Then $\varphi = \varphi(t,r)$
satisfies, for $t, r > 0$,
\begin{equation}
  \label{eq:radialbefp}
  \p_t \varphi
  = \p_r^2 \varphi
  + \p_r \left( \left( r - \frac{d-1}{r} \right) \varphi \right)
  + \p_r \left( \frac{1}{r^{d-2}} \varphi^2 \right).
\end{equation}
We notice that, since $f$ is bounded at $v=0$, $\varphi$ satisfies
that $\lim_{r \to 0} r^{-k}\varphi(t, r) = 0$ for all $k < d - 1$. Since
$f$ is radially symmetric and smooth, $\nabla f(t,0) = 0$ and $\varphi$ also
satisfies the boundary condition
\begin{equation*}
  \p_r \left( r^{1-d} \varphi \right) = 0
  \quad \text{ at $r=0$.}
\end{equation*}
Expanding the derivative one sees that in particular $\lim_{r \to 0} \p_r \varphi(t,r) = 0$. We now define $\Phi(t,r)$ as the indefinite integral of $\varphi$:
\begin{equation}
  \label{ch1}
  \Phi(t,r):=\int_0^r \varphi(t,s) \d s.
\end{equation}
That is, $d \omega_d \Phi$ is the mass of $f$ in a ball of radius $r$,
where $\omega_d$ is the volume of the $d$-dimensional unit ball. One
sees that $\Phi$ satisfies
\begin{equation}
  \label{eq:varg} 
  \left\{
    \begin{array}{l}
      \displaystyle{
        \p_t \Phi = \p_r^2 \Phi + \left(r-\frac{d-1}{r}\right) \p_r \Phi
        + r^{2-d}\left|\p_r \Phi \right|^2},
      \\
      \\
      \displaystyle{
        \Phi(t,\infty)= {\frac{m}{d \omega_d}}}.
    \end{array}
  \right.
\end{equation}
From now on, fix the dimension $d=2$. Then the coefficient of the last
term on the right hand side of equation~\eqref{eq:varg} becomes
independent of $r$ and we can use the Hopf-Cole transformation. We
introduce the new quantity
\begin{equation}
  \label{ch2}
  \Psi(t,r) = e^{\Phi(t,r)}-1.
\end{equation}
Then, equation~\eqref{eq:varg} becomes:
\begin{equation}
  \label{eq:varh}
  \left\{
    \begin{array}{l}
      \displaystyle{
        \p_t \Psi = \p_r^2 \Psi
        +
        \left(r-\frac{1}{r}\right)
        \p_r \Psi,
      }
      \\
      \\
      \displaystyle{\Psi(t,\infty)= {e^{m/d\omega_d}-1,}
      }
    \end{array}
  \right.
\end{equation}
in which the quadratic term has been absorbed by the second derivative of $\Psi$ with respect to $r$. We finally introduce the change
\begin{equation}
  \label{ch3}
  \psi(t,r) := \partial_r \Psi(t,r),
\end{equation}
which leads us to
\begin{equation}
  \label{eq:varu}
  \p_t \psi
  =
  \p_r^2 \psi + \p_r \left( \left( r - \frac{1}{r} \right) \psi \right).
\end{equation}
At this point we only need to notice that equation~\eqref{eq:varu} coincides with the radial form of the 2D linear Fokker-Planck (FP) equation, namely
\begin{subequations}
  \label{fp}
  \begin{align}
  \label{fp-pde}
  \p_t g = \Delta g + \dv(v g),
  &\qquad t > 0, \ v \in \R^2,
  \\
  \label{fp-ic}
  g(0,v) = g_0(v),
  &\qquad v \in \R^2.
  \end{align}
\end{subequations}
More specifically, if $g$ is radially
  symmetric and satisfies \eqref{fp}, then the function
  $\psi =\psi(t,r)$ defined by $\psi(t,|v|)=|v| g(t,v)$ solves
  \eqref{eq:varu}. Hence, all the existing theory for equation
  \eqref{fp} \cite{Risken, frank2006nonlinear} can be used and
translated back to the BEFP equation \eqref{befp} through the
different changes of variable. This gives in particular an
analytic expression for radially symmetric solutions of \eqref{befp}
and directly yields an existence and uniqueness theory for the (BEFP)
equation \eqref{befp} in dimension $2$, as well as information about
the asymptotic behavior of its solutions.

The above changes of variable are summarized in the following diagram:
\begin{equation*}
  \xymatrixcolsep{3.5pc}
  \xymatrix{
    f \ar[r]^{rf = \varphi(r)} 
    &
    \varphi \ar[r]^{\int_0^r \varphi}
    &
    \Phi \ar[r]^{e^\Phi-1}
    &
    \Psi \ar[r]^{\partial_r \Psi}
    &
    \psi \ar[r]^{rg = \psi(r)}
    &
    g \ar@/^1.5pc/[lllll]^{f = \Lambda(g)}
  }
\end{equation*}

\medskip The rest of the paper is organized as follows. In
Section~\ref{changevar} we make precise the translation between
solutions to the Bose-Einstein-Fokker-Planck equation \eqref{befp} and
solutions to the linear Fokker-Planck equation \eqref{fp}. Next, in
Section~\ref{radials} we present our main results on radially
symmetric solutions, Theorems~\ref{thm:BEFP-radial}
and~\ref{thm:befp-asymptotic-radial}. Then, in Section~\ref{general}
we extend these results to more general solutions by comparison with
radial solutions, and state our main { results} for non-radially
symmetric solutions in Theorems~\ref{thm:BEFP}
and~\ref{thm:BEFP-asymptotic}. We state the Csisz\' ar-Kullback
inequality for the Bose-Einstein-Fokker-Planck entropy in
Lemma~\ref{lem:ck}, which is needed for the proof of
Theorem~\ref{thm:BEFP-asymptotic}.


\section{The change of variables}\label{changevar}

We will now make precise the equivalence between radially symmetric
solutions of the FP equation \eqref{fp} and those of the BEFP equation
\eqref{befp}. This will hold for classical solutions of BEFP, as given
in the next definition. We denote by $\mathcal{C}^1_t\mathcal{C}^2_v$
the space of functions $f = f(t,v)$ on $(0,\infty) \times \R^2$ for
which $f$, $\partial_t f$, and all partial derivatives of order up to
$2$ with respect to $v$ exist and are continuous on
$(0,\infty) \times \R^2$.
A \textsl{measure} is always understood to be a non-negative Borel measure with finite total mass {and the space of such measures is denoted by $\mathcal{M}_+(\mathbb{R}^2)$}.

\begin{dfn}[Classical solution to BEFP]
  \label{dfn:befp-solution}
  Let $f_0$ be a nonnegative Borel measure in $\mathbb{R}^2$ (not
  necessarily finite). We say that a function $f$ defined on
  $(0,\infty) \times \R^2$ is a \emph{classical solution} to the
  Cauchy problem {for the BEFD equation} \eqref{befp} with
  initial condition $f_0$ when $f = f(t,v)$ belongs to
  $\mathcal{C}^1_t \mathcal{C}^2_v$, equation \eqref{befp-pde} holds
  in $(0,\infty) \times \R^2$, and $\lim_{t \to 0} f(t) = f_0$ in the
  weak-$*$ sense of measures on $\R^2$.
\end{dfn}

We define a classical solution to the Cauchy problem for the FP equation \eqref{fp} analogously. Note that classical solutions of the FP equation are known to be unique for measures with finite total mass as initial data via the change of variables to the heat equation, see next section and \cite{CT98}. However, since our change of variables works for locally finite measures, we prefer to state the concept in this more general setting. 

\medskip

The change of variables \eqref{ch1}, \eqref{ch2}, and \eqref{ch3} can be summarized as follows: the non-negative, radially symmetric functions
$f$ and $g$ on $\R^2$ are related by
\begin{equation}
  \label{eq:radial-functions}
  |v| f(t,v) = \varphi(t,|v|),
  \quad
  |v| g(t,v) = \psi(t,|v|),
  \qquad t > 0, \ v \in \R^2 ,
\end{equation}
with
\begin{equation}
  \label{transform-phi-u}
  \varphi(t,r)
  := \p_r \log \left( 1 + \int_0^r \psi(t,s)\d s \right)
  \quad \text{ for } t > 0,\ r \geq 0 ,
\end{equation}
or equivalently
\begin{equation}
  \label{transform-u-phi}
  \psi(t,r)
  := \p_r
    \exp\left( \int_0^r \varphi(t,s)\d s\right)
  \quad \text{ for } t > 0,\ r \geq 0 .
\end{equation}
This is also equivalent to
\begin{equation}
  \label{eq:transform-equiv}
  \varphi(t,r)
  = \frac{\psi(t,r)}{1 + \int_0^r \psi(t,s) \d s},
  \qquad
  \psi(t,r)
  = \varphi(t,r) \exp \left( \int_0^r \varphi(t,s)\d s \right),
\end{equation}
which may be a more convenient form for some calculations.

\medskip 

We may also adapt the change of variables
\eqref{eq:radial-functions}--\eqref{transform-u-phi} to work on locally finite non-negative measures (i.e., measures which are finite on any compact set), so that we can consider a measure {as} initial condition. If $f$ and $g$ are non-negative, locally finite, radially
symmetric measures on $\R^2$ then one may define measures $\varphi$ and $\psi$ on $[0,\infty)$ by
\begin{equation}
  \label{eq:ell-psi}
  \varphi([0,r]) := \frac{1}{2\pi} f \left(\overline{B(0,r)}\right),
  \quad
  \psi([0,r]) := \frac{1}{2\pi} g \left(\overline{B(0,r)}\right),
  \qquad r \geq 0.
\end{equation}
Conversely, if $\varphi$ and $\psi$ are radially symmetric measures on $[0,\infty)$, then there are unique radially symmetric measures $f$ and $g$ on $\R^2$ which satisfy \eqref{eq:ell-psi}. When $f$ and $g$ are
functions, $\varphi$ and $\psi$ are just defined by
\eqref{eq:radial-functions}.

The change \eqref{transform-phi-u}--\eqref{transform-u-phi} may be extended to non-negative, locally finite measures by setting
\begin{align}
  \label{eq:transform-psi-ell}
  &\varphi([0,r])
  = \log \left( 1 + \psi([0,r]) \right)
  \quad \text{ for } r \geq 0,
  \\
  \label{transform-ell-psi}
  &\psi([0,r])
  = 
  \exp\left( \varphi([0,r]) \right) - 1
  \quad \text{ for } r \geq 0.
\end{align}
One can check that the change
\eqref{eq:ell-psi}--\eqref{transform-ell-psi} coincides with
\eqref{eq:radial-functions}--\eqref{transform-u-phi} when $f$ and $g$
(and hence also $\varphi$ and $\psi$) are locally integrable
functions.

Hence, by setting $\Lambda(g) = f$ where $f$ is related to $g$ through $\varphi$ and $\psi$, see equations~\eqref{eq:ell-psi}--\eqref{transform-ell-psi}, we define a transformation $\Lambda$ on the set of non-negative, radially symmetric, locally finite measures on $\R^2$.  We observe that from \eqref{eq:transform-psi-ell}--\eqref{transform-ell-psi} one easily deduces that $\Lambda$ and $\Lambda^{-1}$ are continuous in the weak-$*$ topology of measures. We will often consider $\Lambda$ restricted to the set $\mathcal{M}_{+,\text{rad}}(\R^2)$ of non-negative, radially symmetric, finite measures on $\R^2$. In that set, the masses (total variation) of $f$ (denoted by $m$) and $g$ (denoted by $M$) are related by
\begin{equation}
  \label{eq:Mm}
  m
  =
  2 \pi \log\left( 1+ \frac{M}{2 \pi} \right),
  \qquad
  M = 2 \pi \left( e^{m/2\pi} - 1\right).
\end{equation}
The transformation $\Lambda$ is also continuous in $L_{\text{rad}}^1(\R^2)$, the set of integrable and radially symmetric functions in $\R^2$:

\begin{lem}
  \label{lem:L1-cont}
  The change $\Lambda$ is a bijection in $L_{\text{rad}}^1(\R^2)$,
  and both $\Lambda$ and $\Lambda^{-1}$ are locally Lipschitz
  continuous in the $L^1$-norm.
\end{lem}

\begin{proof}
  That $\Lambda$ is a bijection in $L_{\text{rad}}^1(\R^2)$ is clearly seen from \eqref{eq:transform-equiv}, for example. Take $g_1$ and  $g_2$ in {$L_{\text{rad}}^1(\R^2)$} and set $f_i = \Lambda(g_i)$, $|v| \varphi_i(|v|) = f_i(v)$, and $|v| \psi_i(|v|) = g_i(v)$ for $i=1,2$ and $v\in\R^2$. Set $F_i(r) = (1 + \int_0^r
  \psi_i(s) \d s)^{-1}$ for $i=1,2$ and notice that 
  \begin{equation*}
 0 \le F_i(r) \le 1 \;\;\text{ and }\;\;   |F_1(r) - F_2(r)| \leq \int_0^r |\psi_1(s) -
    \psi_2(s)| \d s\ , \quad r\ge 0 .
  \end{equation*}
  Combining these inequalities with \eqref{eq:transform-equiv} gives
  \begin{align*}
    \frac{1}{2\pi} \|f_1 - f_2\|_1
    & =
    \int_{0}^\infty |\varphi_1(r) - \varphi_2(r)| \d r
    =
    \int_{0}^\infty | F_1(r) \psi_1(r) - F_2(r) \psi_2(r)| \d r
    \\
    & \leq
    \int_{0}^\infty
    F_1(r)
    \left| \psi_1(r) - \psi_2(r) \right|  \d r
    +
    \int_{0}^\infty
    \left| F_1(r) - F_2(r) \right|
    \psi_2(r) \d r
    \\
    & \leq
    \int_{0}^\infty
    \left| \psi_1(r) - \psi_2(r) \right| \d r
    +
    \int_{0}^\infty
    \int_0^r \left| \psi_1(s) - \psi_2(s)  \right| \psi_2(r)
    \d s \d r
    \\
    &\leq \frac{1}{2\pi} \|g_1-g_2\|_1
    + \frac{1}{4\pi^2} \|g_1-g_2\|_1 \|g_2\|_1,
  \end{align*}
which shows that $\Lambda$ is locally Lipschitz continuous in the $L^1$-norm. A very similar argument shows that $\Lambda^{-1}$ is also locally Lipschitz continuous.
\end{proof}

This change of variables is also very regular, in the sense that it
preserves the smoothness of the function to which it is applied. This will enable us later to translate properties of the FP equation \eqref{fp} to the BEFP equation \eqref{befp} in dimension~2:

\begin{lem}
\label{lem:C2-cont}
The change $\Lambda$ is a bijection from the set of radially symmetric functions in $\mathcal{C}^2(\R^2)$ to itself (in the sense that $\Lambda(g)$ coincides a.e.~with a function in $\mathcal{C}^2(\R^2)$). In a similar way, if $g \in \mathcal{C}^1_t\mathcal{C}^2_v$ with $g(t)$ radially symmetric for each $t > 0$ and we define $f(t,v) = \Lambda(g(t))(v)$ then $f \in \mathcal{C}^1_t\mathcal{C}^2_v$. 
\end{lem}

Before proving this we mention an elementary lemma that we wish to
make explicit. We denote by $\mathcal{C}^2([0,\infty))$ the space of twice differentiable real functions on $[0,\infty)$ such that all
derivatives of order up to $2$ have limits at $0$.

\begin{lem}[Regularity of a radial function]
  \label{c2radial}
Take $\phi\in\mathcal{C}^2([0,\infty))$ such that
  \begin{equation}
    \lim_{r \to 0} \phi'(r) = 0 . \label{pim}
  \end{equation}
  If we define a radial function $f \colon \mathbb{R}^d \to
  \mathbb{R}$ by $f(v) = \phi(|v|)$, then $f\in
  \mathcal{C}^2(\mathbb{R}^d)$. Conversely, if $f\in \mathcal{C}^2(\mathbb{R}^d)$ is a radially symmetric function, then the function $\phi \colon [0,\infty) \to \mathbb{R}$ defined by $\phi(r) := f(r,0,\ldots,0)$ belongs to $\mathcal{C}^2([0,\infty))$ and satisfies \eqref{pim}.
\end{lem}

\begin{rem}
  It is easy to extend this lemma also to time-dependent functions:
  for example, if a function $\phi = \phi(t,r)$ is in
  $\mathcal{C}^1_t\mathcal{C}^2_r((0,\infty) \times [0,\infty))$ and
  satisfies {
  \begin{equation*}
    \lim_{(s,r) \to (t,0)} \partial_r \phi(s,r) = 0 
  \end{equation*} }
  for all $t > 0$, then $f(t,v) := \phi(t,|v|)$ is in
  $\mathcal{C}^1_t\mathcal{C}^2_v$.
\end{rem}

\begin{proof}[Proof of Lemma \ref{c2radial}]
Clearly $f\in \mathcal{C}^2(\R^d\setminus\{0\})$ and we only need to check that the first and second-order derivatives of $f$ can be extended continuously at $v=0$. For $1\le i,j\le d$ and $v\in \mathbb{R}^d$, $v\ne 0$,
$$\partial_i f(v) = \phi'(|v|) \frac{v_i}{|v|}\ , \quad \partial_{i} \partial_j f(v) = \delta_{ij}\phi''(|v|)
  + \left(\frac{v_iv_j}{|v|^2}-\delta_{ij}\right)
  \left(\phi''(|v|) - \frac{\phi'(|v|)}{|v|}\right)\ .
$$
Since $v\mapsto v_i/|v|$ and $v\mapsto v_i v_j/|v|^2$ are bounded and
$$
\lim_{r\to 0} \left( \phi''(r) - \frac{\phi'(r)}{r} \right) = \phi''(0)-\phi''(0)=0
$$
by \eqref{pim} we conclude that
$$
\lim_{v\to 0} \partial_i f(v) = 0 \;\;\text{ and }\;\; \lim_{v\to 0} \partial_i \partial_j f(v) = \phi''(0) \delta_{ij}\ .
$$
Consequently, $f$ can be extended to a $\mathcal{C}^2$-smooth function in $\R^d$ by setting $\nabla f(0)=0$ and $D^2 f(0) = \phi''(0) \mathrm{id}$.

The converse is obvious as the $\mathcal{C}^2$-regularity of $\phi$ readily follows from that of $f$ while the radial symmetry of $f$ guarantees that $\nabla f(0)=0$ and thus \eqref{pim}.
\end{proof}

\begin{proof}[Proof of Lemma~\ref{lem:C2-cont}]
  Take a radially symmetric function $g \in \mathcal{C}^2(\R^2)$ and let us show that $f = \Lambda(g)\in \mathcal{C}^2(\R^2)$ (the proof for $\Lambda^{-1}$ is similar). Clearly $f$ is $\mathcal{C}^2$-smooth in $\R^2 \setminus \{0\}$, since the change of variables has a smooth expression away from $v = 0$. In order to see that $f$ is twice differentiable at $v=0$, we use \eqref{eq:radial-functions} and \eqref{eq:transform-equiv} to write $f$ as
  $$f(v) = \frac{\varphi(|v|)}{|v|} 
  = \frac{\psi(|v])}{|v|} \left( 1+\int_0^{|v|}\psi(s)ds \right)^{-1}
  = g(v) F(|v|),
$$ 
where we have set
  \begin{equation*}
    F(r) := \left( 1+\int_0^{r}\psi(s)ds \right)^{-1}\ , \quad r\ge 0.
  \end{equation*}
  Observe that $\psi(r) = r g(r,0,\ldots,0)$ for $r\ge 0$ and thus belongs to $\mathcal{C}^2([0,\infty))$ with $\psi(0)=0$. then $F'(0)=0$ and we deduce from Lemma~\ref{c2radial} that $f$ extends to a $\mathcal{C}^2$-smooth function in $\R^2$. The time-dependent results can be proved with very little modifications.
\end{proof}

\medskip
We have the following result:

\begin{thm}[Equivalence between BEFP and FP]
  \label{thm:change}
  Let $f_0$ and $g_0$ be non-negative radially symmetric measures on
  $\R^2$ related by \eqref{eq:ell-psi}--\eqref{transform-ell-psi}, and
  let $f$, $g$ be non-negative radially symmetric functions on
  $(0,\infty) \times \R^2$ related by
  \eqref{eq:radial-functions}--\eqref{transform-u-phi}. The following
  are equivalent:
  \begin{enumerate}
  \item \label{it:befp} $f$ is a
  classical solution
    to the Bose-Einstein-Fokker-Planck equation \eqref{befp} with
    initial condition $f_0$.
  \item \label{it:fp} $g$ is a classical solution to the
    linear Fokker-Planck equation \eqref{fp} with initial condition $g_0$.
  \end{enumerate}
\end{thm}

\begin{proof}
We just show one of the implications, since the reverse one is analogous. Assume point~\eqref{it:fp}, and let us show point~\eqref{it:befp}. Since $g$ is a classical solution to \eqref{fp}, it lies in $\mathcal{C}_t^1\mathcal{C}_v^2$. Hence the calculations \eqref{eq:radialbefp}--\eqref{eq:varu} are rigorous and we see that $f$ satisfies \eqref{befp} in $(0,\infty) \times (\R^2 \setminus \{0\})$ (the point $v=0$ being a possible problem due to the radial change of variables.) But due to Lemma~\ref{lem:C2-cont},
  $f$ belongs to $\mathcal{C}_t^1\mathcal{C}_v^2$ and a continuity argument entails that $f$ satisfies \eqref{befp-pde} in $(0,\infty) \times \R^2$. Finally, $f$ satisfies the initial condition in the weak-$*$ sense since the change $\Lambda$ is continuous for the weak-$*$ convergence of measures.
\end{proof}

Notice that some of the classical solutions given by Theorem~\ref{thm:change} are not covered by any of the existence results in the literature. The change of variables gives us classical solutions to a nonlinear equation with a measure as initial condition,  which is quite remarkable. Finally, let us point out that radially symmetric classical solutions to the FP equation \eqref{fp} are unique as soon as some growth condition at infinity is satisfied. Since the FP equation \eqref{fp} is equivalent to the heat equation as we will see next, then uniqueness is satisfied for initial data in tempered distributions (classical solutions are distributional solutions), in particular for finite mass non-negative measures (which are our main interest).


\section{Global radially symmetric solutions}
\label{radials}

The Fokker-Planck equation \eqref{fp} is equivalent to the heat equation via a self-similar change of variables. Moreover, it can be explicitly solved through the Fourier transform and its fundamental solution can be explicitly found leading to
$$
\calF(t,v,w):=a(t)^{-d/2}M_{\nu(t)}(a(t)^{-1/2}v-w)
$$
with
$$
a(t):=\rme^{-2t}\quad\text{,}\quad\nu(t):=\rme^{2t}-1\quad\text{and}\quad
M_{\lambda}(\xi):=(2\pi \lambda)^{-d/2}\rme^{-|\xi|^2/2\lambda}
$$
for any $\lambda>0$. Let us focus again on the two dimensional case and define the operator $\calF[\xi]$
acting on functions $\xi$ by:
\begin{equation}\label{eq:fpFonSol}
\calF[\xi(w)](t,v) := \int_{\R^2}\calF(t,v,w)\xi(w)\,\rmd w \ , \quad (t,v)\in (0,\infty)\times \R^2\ .
\end{equation}
The following result is well-known, see for example \cite[Lemma~3.1]{CT98}:

\begin{thm}[Solutions to FP]
  \label{thm:FP}
  Let $S'(\R^2)$ denote the space of tempered distributions. Then, given $g_0 \in S'(\R^2)$ there exists a unique weak solution $g\in \mathcal{C}([0,\infty), S'(\R^2))$ to the linear Fokker-Planck equation \eqref{fp} with initial condition $g_0$ given by $g=\calF[g_0]$. In addition, $g$ is smooth on $(0,\infty)\times \R^2$, and positivity and radial symmetry are conserved; more
  precisely:
  \begin{enumerate}
  \item If $g_0$ is not identically equal to zero then $g(t,v) > 0$ for all $t > 0$ and  $v \in \R^2$.
  \item If $g_0$ is radially symmetric then $g(t)$ is
    radially symmetric for all $t > 0$.
  \end{enumerate}
\end{thm}

\begin{rem}\label{rem:thm:FP}
  Since $\mathcal{M}_+(\R^2)$ endowed with the weak-$*$ topology is
  continuously embedded in $S'(\R^2)$, in our context this result
  tells us that, indeed, for any $g_0 \in \mathcal{M}_+(\R^2)$ there
  exists a unique classical solution $g \in \mathcal{C}([0,\infty),
  \mathcal{M}_+(\R^2))$ of the FP equation \eqref{fp} (where we always
  consider the weak-$*$ topology on $\mathcal{M}_+(\R^2)$) and this
  solution satisfies the above properties.
\end{rem}

Finally, let us recall some moment and smoothing properties
of the solution of the FP equation \eqref{fp}. With this aim, let us
introduce the spaces $L^p_\ell(\R^2)$ with norms
$$
\|f\|_{L_{\ell}^p}:=\|(1+|v|^{\ell})f\|_p\quad \text{and}
\quad\|f\|_p:=\left(\int_{\R^d}|f|^p\rmd
v\right)^{\frac{1}{p}}\text{.}
$$
for $1\leq p\leq \infty$ and ${\ell}\geq 0$, with the usual essential
supremum norm definition for the case $p = \infty$. Solutions {$g$} to
equation~\eqref{fp} satisfy
\begin{equation*}\label{smooth}
\|\p_{\a}g(t)\|_{L^p_{\ell}}\leq
\frac{C e^{\left(\frac{2(p-1)}{p}+|\a|\right)t}}{\nu(t)^{\left(\frac{1}{q}-\frac{1}{p}\right)
+\frac{|\a|}{2}}}\|g_0\|_{L^q_{\ell}} ,
\end{equation*}
for all $1\leq q\leq p\leq \infty$, $\ell \geq 0$, $\a \in \N^d$ a
multiindex, and $t>0$. Here, $C > 0$ is a number depending on
$p,q,\ell$ and $\alpha$. For a proof of this result see
\cite[Proposition~A.1]{clr08}. In particular, the case $\alpha=0$ and
$p=q$ gives
\begin{equation}\label{unif}
\|g(t)\|_{L^p_{\ell}}\leq
C e^{\frac{2(p-1)t}{p}} \|g_0\|_{L^p_{\ell}} 
\end{equation}
for all $1\leq p\leq \infty$, ${\ell}\geq 0$, and $t>0$ with $C$
depending only on $p$ and ${\ell}$. The case $\alpha = 0$, $q=1$ gives
\begin{equation}\label{unif1}
  \|g(t)\|_{L^p_{\ell}}
  \leq
  C \left( \frac{e^{2t}}{e^{2t} - 1} \right)^{(p-1)/p}  \|g_0\|_{L^1_{\ell}}\ , \quad t>0\ .
\end{equation}

\medskip
Theorem~\ref{thm:change} immediately implies an analogous result for \emph{radially symmetric} solutions to the BEFP equation \eqref{befp}:

\begin{thm}[Radial solutions to BEFP]
  \label{thm:BEFP-radial}
  For any radially symmetric $f_0 \in \mathcal{M}_+(\R^2)$ there is a
  unique classical solution
  $f \in \mathcal{C}([0,\infty), \mathcal{M}_+(\R^2))$ to the BEFP
  equation \eqref{befp} with initial condition $f_0$ and it is
    given by $f(t) = \Lambda(g(t))$ for all $t\ge 0$ where $g$ is the
    solution to the FP equation \eqref{fp} with initial condition
    $g_0 := \Lambda^{-1}(f_0)$. In addition, $f$ is bounded in
    $L^\infty((\delta,\infty)\times \R^2)$ for each $\delta>0$ and
    $f(t)$ is radially symmetric for all $t\ge 0$. Furthermore, if
  $f_0$ is not identically equal to zero then $f(t,v) > 0$ for
  all $t > 0$ and $v \in \R^2$.
\end{thm}

\begin{rem}
  This solution is also the unique mild solution in the sense of \cite[Eq.~(2.2)]{clr08} for initial data in the class considered
    therein, since the authors prove that mild solutions are unique
  up to the maximal existence time. Of course, the above
  theorem also shows that the unique mild solution does not blow up
  and is defined for all positive times.
\end{rem}

\begin{proof}
  Consider the solution $g$ to the FP equation \eqref{fp} with initial
  condition $\Lambda^{-1}(f_0)$ given by Theorem~\ref{thm:FP}. This
  solution is also in $\mathcal{C}([0,\infty), \mathcal{M}_+(\R^2))$
  (see Remark~\ref{rem:thm:FP}). Since $g$ is a classical solution to
  the FP equation \eqref{fp}, the function $f := \Lambda(g)$ is a
  classical solution to the BEFP equation \eqref{befp} by
  Theorem~\ref{thm:change}. It also lies in
  $\mathcal{C}([0,\infty), \mathcal{M}_+(\R^2))$, since $\Lambda$
  preserves this space and is continuous with respect to the weak-$*$
  topology of measures. Finally, the claimed boundedness in
    $(\delta,\infty)\times \R^2$ readily follows from \eqref{unif1}.
\end{proof}

We next give another example of the results that one can get by translating known results for the FP equation \eqref{fp}, which deals with the large time behavior. The following result is well known and can be found for example in \cite[Theorem~4.1]{CT98}.

\begin{thm}[Asymptotic behavior for FP]
Take $g_0 \in L^1_+(\R^2)$ and denote the total mass of $g_0$ by $M$. The solution $g$ of the FP equation \eqref{fp} with initial condition $g_0$ satisfies
  \begin{equation*}
    \| g(t) - M g_\infty \|_{L^1}
    \leq
    e^{-t} \| g_0 - M g_\infty \|_{L^1}
    \quad \text{ for all $t \geq 0$,}
  \end{equation*}
  where the equilibrium $g_\infty$ is given
by
\begin{equation*}
  g_\infty(v) = \frac{1}{2\pi} e^{-|v|^2/2},
  \qquad v \in \R^2 .
\end{equation*}
\end{thm}

Notice that if we apply the change \eqref{eq:radial-functions}--\eqref{transform-phi-u} to $M g_\infty$ we obtain precisely $f_\infty^\beta$, with
\begin{equation}
  \label{eq:beta-mass}
  \beta = \frac{2\pi}{M} + 1.
\end{equation}
Thanks to Theorem~\ref{thm:change} and Lemma~\ref{lem:L1-cont}, this
translates into the following result for the BEFP equation \eqref{befp}
\begin{thm}[Asymptotic behavior for radial BEFP]
  \label{thm:befp-asymptotic-radial}
  Take a radially symmetric $f_0 \in L^1_+(\R^2)$ and denote its mass
  by $m$. There is a constant $K = K(m)>0$, depending only on $m$,
  such that the solution $f$ of the BEFP equation \eqref{befp} with
  initial condition $f_0$ satisfies
  \begin{equation*}
    \| f(t) - f_\infty^\beta \|_{L^1}
    \leq
    K e^{-t} \| f_0 - f_\infty^\beta \|_{L^1}
    \quad \text{ for all $t \geq 0$},
  \end{equation*}
  choosing $\beta$ to be such that $f_\infty^\beta$ has mass $m$:
  \begin{equation}
    \beta
    = \left( 1 - e^{-m / (2\pi)} \right)^{-1} . \label{eq:beta-mass-b}
  \end{equation}
\end{thm}

\begin{rem}
  The constant $K(m)$ above can be taken to be
  \begin{equation*}
    K(m) = \left( \Lip_M(\Lambda) \right) \left( \Lip_m (\Lambda^{-1}) \right),
  \end{equation*}
  where $\Lip_R$ denotes the $L^1$-Lipschitz constant in the
  $L^1$-ball of radius $R$ (see Lemma~\ref{lem:L1-cont}.)
\end{rem}

Another interesting consequence of Theorem~\ref{thm:change} is that
radially symmetric solutions to the BEFP equation \eqref{befp} have an explicit expression in
terms of the initial condition $f_0$, obtained by transforming the
explicitly computable solution $g$ of the FP equation \eqref{fp} through the
change of variables
\eqref{eq:radial-functions}--\eqref{eq:transform-equiv}:
\begin{equation*}
 f(t,v)
  = \frac{g(t,v)}{1 + \int_0^{|v|} s \psi(t,s) \d s}, \quad (t,v)\in (0,\infty)\times \R^2\ ,
\end{equation*}
with $\psi$ related to $g$ by \eqref{eq:radial-functions}. This relation obviously implies that, for $(t,v)\in (0,\infty)\times \R^2$, 
\begin{equation*}
\frac{2\pi}{2\pi+\|g_0\|_{L^1}} g(t,v) \leq f(t,v) \leq g(t,v) \le f(t,v) e^{\|f_0\|_{L^1}/2\pi}
\end{equation*}
by the conservation of the $L^1$-norm. This readily translates bounds
for $g$ into bounds for $f$; indeed, on the one hand, equation~\eqref{unif} gives
\begin{equation}
  \label{unif2}
  \|f(t)\|_{L^p_{\ell}}
  \leq
  C e^{\frac{2(p-1)t}{p}} \|g_0\|_{L^p_{\ell}}
  \leq
  C e^{\frac{2(p-1)t}{p}}\ e^{\|f_0\|_{L^1}/2\pi}\ \|f_0\|_{L^p_{\ell}}
\end{equation}
for all $1\leq p\leq \infty$, $\ell \geq 0$, and $t>0$ with $C$ depending only on $p$ and $\ell$. On the other hand, equation~\eqref{unif1} gives 
\begin{eqnarray}
 \|f(t)\|_{L^p_{\ell}}
 & \leq &
  C {\left( \frac{e^{2t}}{e^{2t}-1} \right)^{(p-1)/p}} \|g_0\|_{L^1_{\ell}} \nonumber \\  
& \leq &  C {\left( \frac{e^{2t}}{e^{2t}-1} \right)^{(p-1)/p} e^{\|f_0\|_{L^1}/2\pi}}\ \|f_0\|_{L^1_{\ell}}.  \label{unif4}
\end{eqnarray}
These bounds on $f$ will be used in the next section
in order to show that solutions to \eqref{befp} starting from general inital data are globally defined in time.

\medskip

An interesting particular solution is the transform of the
solution to the FP equation \eqref{fp} with initial condition the Dirac mass $\delta_0$ at $v=0$ which may be referred to as a \textsl{fundamental solution} for the (nonlinear) BEFP equation \eqref{befp}:
\begin{equation*}
  f(t,v)
  =
  \frac{1}{\vartheta(t)}
  \left[
    (2\pi + 1) e^{|v|^2/2\vartheta(t)} - 1
  \right]^{-1},
  \quad \text{ for $v \in \R^2$, $t \geq 0$},
\end{equation*}
with
\begin{equation*}
  \vartheta(t) := (1-e^{-2t})
  \quad \text{ for $t \geq 0$}.
\end{equation*}

  There is also an explicit infinite-mass solution to the BEFP
  equation \eqref{befp} obtained by transforming the solution $g(t,v) = A e^{2t}$, $A>0$,
  to the FP equation \eqref{fp}:
  \begin{equation*}
    f(t,v)
    =
    {2} \left( {2} A^{-1} e^{-2t} + |v|^2 \right)^{-1},
    \quad \text{ for $v \in \R^2$, $t \geq 0$}.
  \end{equation*}
  Strictly speaking, this solution does not satisfy Definition
 ~\ref{dfn:befp-solution} because it is not integrable; however, it
  satisfies \eqref{befp-pde} for all $(t,v) \in (0,\infty) \times
  \R^2$. This solution converges exponentially fast to an
  infinite-mass stationary solution $f_\infty^*(v) := |v|^{-2}$.


\section{Global (non-radially symmetric) solutions}
\label{general}

In the case of Fermi-Dirac interactions, where the nonlinear drift in \eqref{befp} is negative, a local well-posedness theory and $L^1$-contraction and maximum principles were proved in \cite{clr08}. However, the proofs are totally independent of the sign of the drift and can be directly translated to our case:

\begin{prp}[Summary of results from \cite{clr08}]
  \label{befp-results}
  Let $f_0 \in (L^1\cap L^p_{\ell})(\R^2)$ be non-negative, with $p>2$ and $\ell\geq 1$. Then:
  \begin{description}\label{prp:sumCLR08}

  \item[Local Existence] There exists a maximal time
      $T\in (0,\infty]$ depending only on $\|f_0\|_{L_{\ell}^p}$ and
      $\|f_0\|_{L^1}$ such that there is a unique solution $f$ in
    $\mathcal{C}([0,T),(L_{\ell}^p\cap L^1)(\R^2))$ of the BEFP
    equation \eqref{befp} which is non-negative and satisfies
  \begin{equation*}
  \|f(t)\|_{L^1} = \|f_0\|_{L^1}\ , \qquad t\in [0,T)\ . 
  \end{equation*} 
  If $T<\infty$ then $\|f(t)\|_{L_{\ell}^p}\longrightarrow \infty$ as
$t\to T$.

\item[Contractivity and Comparison Principle] Let
    $T\in (0,\infty]$ and $f_1$, $f_2$
    $\in \mathcal{C}([0,T),(L_{\ell}^p\cap L^1)(\R^2))$ be two
    solutions to the BEFP equation \eqref{befp} with initial data
    $f_{1,0}$ and $f_{2,0}$ in $(L_{\ell}^p\cap L^1)(\R^2)$,
    respectively. Then
    $$
    \|f_1(t) - f_2(t)\|_{L^1} \leq \|f_0 - g_0\|_{L^1}\ , \quad t\in [0,T)
    $$
    Furthermore, if $f_{1,0}\leq f_{2,0}$ then $f_1(t,v)\leq f_2(t,v)$
    for all $t\in [0,T)$ and $v\in \R^2$.
  \end{description}
\end{prp}

Thus one can deduce the following result from Theorem~\ref{thm:BEFP-radial}
and {Proposition~\ref{prp:sumCLR08}}:

\begin{thm}
  \label{thm:BEFP}
  Take a non-negative initial condition $f_0 \in
  L^\infty_{\ell}(\R^2)$ with ${\ell}>2$. Then there is a unique
  solution $f \in \mathcal{C}([0,\infty), L^\infty_{\ell}(\R^2))$ (in
  the sense of \cite{clr08}) of the BEFP equation \eqref{befp} with
  initial condition $f_0$. In addition, if $\|f_0\|_{L^1_\ell} <
 \infty$, there is some $C > 0$ which depends only on $f_0$ such
  that
  \begin{equation}
  \| f(t) \|_{ L^1_{\ell}} +  \| f(t) \|_{ L^\infty_{\ell}}
    \leq C
    \quad \text{ for all $t \ge 1$.} \label{poum}
  \end{equation}
\end{thm}
\begin{proof}[Proof]
  The local well-posedness is provided by Proposition~\ref{befp-results},
  so we only need to show that the norm $\|f(t)\|_{L^\infty_\ell}$
  does not blow up in finite time. Let $T$ be the maximal existence time of $f$ given by Proposition~\ref{befp-results}. We consider the radially symmetric
  solution $f^\#$ to \eqref{befp} corresponding to the initial condition
  $$
  f_0^\#(v)= \frac{\|f_0\|_{L^\infty_{\ell}}}{1+|v|^{\ell}}\ , \quad v\in\R^2\ ,
  $$
which we know to be globally defined by Theorem~\ref{thm:BEFP-radial}, since $f_0^\#$ has finite mass due to the condition $\ell > 2$. Clearly, $f_0 \leq f_0^\#$ and by the
  comparison principle in Proposition \ref{befp-results} we know that
  $f(t,v) \leq f^\#(t,v)$ for all $t \in [0,T)$ and $v \in \R^2$. Equation~\eqref{unif2} with $p=\infty$ gives us
  \begin{equation*}
    \|f(t)\|_{L^\infty_\ell}
    \leq
    \|f^\#(t)\|_{L^\infty_\ell}
    \leq
    C e^{2t} \|f_0^\#\|_{L^\infty_\ell}
  \end{equation*}
  for some $C>0$ which depends only on $\ell$
  and $\|f_0\|_{L^1}$. This estimate prevents the $L^\infty_\ell$-norm of $f(t)$ to explode in finite time, so the
  solution must in fact be globally defined, that is $T=\infty$. In order to obtain the additional statement \eqref{poum} it suffices to apply equation~\eqref{unif4} with $p=1$ and $p=\infty$.
\end{proof}

\medskip

We now turn to the large time behavior of solutions to \eqref{befp} and recall that equation~\eqref{befp}, in any dimension $d$, has an entropy functional
given by
\begin{equation}
  \label{eq:entropy}
  H(f) = \int_{\R^d} \left[ \frac{|v|^2}{2} f(v) + f(v) \log f(v) - (f(v)+1) \log (f(v)+1) \right] \d v
\end{equation}
in the following sense: any solution $f$ to \eqref{befp} with a bounded moment of order $\ell> 2$ satisfies
\begin{equation}
  \label{eq:entropy-dissipation}
  \ddt H(f) = - D(f) := - \int_{\R^d} f(1+f)(v)
  \left|
      v + \nabla \log \left( \frac{f}{1+f} \right)(v)
  \right|^2 \d v.
\end{equation}
The availability of this functional leads us to the following result:
\begin{thm}
  \label{thm:BEFP-asymptotic}
  Take a non-negative $f_0 \in L^\infty_\ell(\R^2) \cap L^1_\ell(\R^2)$
  with $\ell > 2$. Then the solution $f$ of the BEFP equation
  \eqref{befp} with initial condition $f_0$ given by Theorem
 ~\ref{thm:BEFP} satisfies
  \begin{equation*}
    \lim_{t\to\infty} \| f(t) - f_\infty^\beta \|_{L^1}
    = 0\ ,
  \end{equation*}
  where $\beta$ is given by \eqref{eq:beta-mass-b} in terms of the mass $m$ of $f_0$.
\end{thm}

\begin{rem}
  Due to Theorem~\ref{thm:BEFP} solutions to \eqref{befp} emanating from non-negative initial data in $L^\infty_\ell(\R^2) \cap L^1_\ell(\R^2)$
  with $\ell > 2$ are bounded in $L^\infty$
  globally in time and a simple interpolation argument shows that for
  all $1 \leq p < \infty$ we also have
  \begin{equation*}
    \lim_{t\to\infty} \| f(t) - f_\infty^\beta \|_p = 0\ .
  \end{equation*}
\end{rem}

\begin{proof}[Proof of Theorem~\ref{thm:BEFP-asymptotic}]
  We follow a common technique to deduce convergence to equilibrium from the existence of a Lyapunov functional such as \eqref{eq:entropy}; see for example \cite{DD,BCCP}.
  
Take a solution $f = f(t,v)$ of the BEFP equation \eqref{befp} with initial condition $f_0$ and recall that Theorem~\ref{thm:BEFP} guarantees that $\{f(t)\ :\ t\ge 1\}$ is bounded in $L^\infty_\ell(\R^2) \cap L^1_\ell(\R^2)$.

  \smallskip
  \noindent
  \textbf{Step 1: Strong convergence of $f(t_n)$ for some
    sequence $\{t_n\}_{n\ge 1}$ of times.} From \eqref{eq:entropy-dissipation} (which we can use since we are assuming $f_0 \in L^1_\ell(\R^2)$ for $\ell>2$)
  and the fact that $H(f(t))$ is bounded below by $H(f_\infty^\beta)$
  (due to Lemma~\ref{lem:ck} below) we see that there must exist a divergent
  sequence $\{t_n\}_{n \geq 1}$, $t_n\ge 1$, such that $D(f(t_n)) \to 0$ as $n \to
  \infty$. Due to \eqref{poum} we know that there exist a subsequence of
  $\{t_n\}_{n \geq 1}$ (which we still call $\{t_n\}$ by an abuse of
  notation) and a non-negative function $f_\infty \in L^1(\R^2) \cap
  L^\infty(\R^2)$ such that $f(t_n) \rightharpoonup f_\infty$ weakly
  in $L^p(\R^2)$, for all $p \in [1,\infty)$, and also in the
  weak-$*$ topology of $L^\infty(\R^2)$. We also note that the bounds \eqref{poum} ensure that 
\begin{equation}
\int_{\R^2} f_\infty(v) \d v = \lim_{n\to\infty} \int_{\R^2} f(t_n,v) \d v = m := \|f_0\|_{L^1}\ .\label{pom}
\end{equation}

Setting $f_n:=f(t_n)$, $n\ge 1$, let us now see that the convergence of $\{f_n\}_{n\ge 1}$ must in fact be strong, due to the fact that $D(f_n) \to 0$. Indeed, expanding the square
\begin{align*}
    D(f_n) & = \int_{\R^2} f_n(1+f_n)(v) |v|^2 \d v
+ \int_{\R^2} \frac{|\nabla f_n|^2}{f_n(1+f_n)}(v) \d v  
    + 2 \int_{\R^2} v \cdot \nabla f_n(v) \d v \\
 & \geq \int_{\R^2} \frac{|\nabla f_n|^2}{f_n(1+f_n)}(v) \d v 
    - 4 m.
  \end{align*}
  Since $D(f_n) \to 0$, this shows that the sequence $\{\log
  (f_n/(1+f_n))\}_{n \geq 1}$ is bounded in $H^1(\R^2)$. In
  particular, there is a subsequence of $\{f_n\}_{n \geq 1}$ (again denoted by $\{f_n\}$) which converges pointwise a.e. to $f_\infty$ and weak limits coincide. Vitali's theorem then shows that, after possibly extracting a further subsequence, we can conclude that 
\begin{equation*}
    f_n \to f_\infty \quad \text{ strongly in $L^p(\R^2)$ as $n \to \infty$,}
\end{equation*}
for all $1\leq p <\infty$.

  \smallskip
  \noindent
  \textbf{Step 2: Proof that $f_\infty = f_\infty^\beta$.}  Define $h_n := f_n / (1+f_n)$, $n\ge 1$. From the previous step, we know that $h_n$ converges pointwise a.e. to $h_\infty := f_\infty/(1+f_\infty)$. Since $h_n \leq f_n$ is clearly uniformly bounded in all $L^p(\R^2)$ ($p \in [1,\infty]$), there is another subsequence of $\{t_n\}_{n \geq 1}$ such that $\{h_n\}_{n\ge 1}$ converges weakly in all $L^p(\R^2)$, and weakly-$*$ in $L^\infty(\R^2)$ to $h_\infty$. Now we have, by the Cauchy-Schwarz inequality,
  \begin{align*}
    \left( \int_{\R^2} \left| v h_n(v) + \nabla h_n(v) \right|  \d v \right)^2
   & \leq
    \left( \int_{\R^2} h_n(v) \left| v + \frac{\nabla h_n}{h_n}(v) \right|^2 \d v \right) \|h_n\|_{L^1} \\
   & \leq m\, D(f_n) \to 0
  \end{align*}
  as $n \to \infty$. This implies
  that $v h_n + \nabla h_n \to 0$ in $L^1(\R^2)$ as $n \to \infty$ which, together with the already established convergences of $h_n$ to $h_\infty$ implies that $v h_\infty(v) + \nabla h_\infty(v) = 0$ in the sense of distributions and in fact almost everywhere. Consequently, $h_\infty(v) = \tilde{\beta} \exp(-|v|^2/2)$ for a.e. $v\in\R^2$ for some $\tilde{\beta} \ge 0$ and thus $f_\infty = f_\infty^{\tilde{\beta}}$ a.e. in $\R^2$. Finally, $\tilde{\beta} = \beta$ is given by \eqref{eq:beta-mass-b} since the mass of $f_\infty$ is the same as that of $f_0$ by \eqref{pom}.

  \smallskip
  \noindent
  \textbf{Step 3: Convergence of the entropy.} Since $t \mapsto
  H(f(t))$ is nonincreasing and larger than $H(f_\infty^\beta)$ (see
  Lemma~\ref{lem:ck} below) there is $H_\infty \geq H(f_\infty^\beta)$ such
  that
  \begin{equation}
    \label{eq:entropy-converges}
    H(f(t)) \to H_\infty \quad \text{ as $t \to \infty$.}
  \end{equation} 
  From the previous two steps we know that $\{f_n=f(t_n)\}_{n\ge 1}$
  converges pointwisely to $f_\infty^\beta$ while \eqref{poum} ensures
  that $\{f_n\}_{n\ge 1}$ is bounded in
  $(L_\ell^1\cap L^\infty_\ell)(\R^2)$, hence also in
  $L^2_\ell(\R^2)$. Since $\ell>2$ the Dominated Convergence Theorem
  now shows that $H(f_n) \to H(f_\infty^\beta)$, which implies that
  $H_\infty = H(f_\infty^\beta)$.

  Using Lemma~\ref{lem:ck} we directly deduce from \eqref{eq:entropy-converges}
  that $\|f(t) - f_\infty^\beta\|_1 \to 0$ as $t \to \infty$,
  finishing the proof.
\end{proof}

We finally give a Csiszár-Kullback inequality for our entropy
functional, which was needed to show that convergence in entropy
implies convergence in $L^1$ norm in the proof of Theorem
\ref{thm:BEFP-asymptotic}. We would like to warmly thank the anonymous
referee for suggesting an improvement of the following result in the
preprint version of this paper.

\begin{lem}[Csiszár-Kullback-type inequality for the BEFP entropy]
  \label{lem:ck}
  Take a non-negative function $f \in L^1(\R^d)$ with mass $\|f\|_1
  \leq m_\mathrm{c}$, and $\beta > 0$ such that the mass of $f_\infty^\beta$
  is equal to that of $f$. Then
  \begin{equation}
    \label{eq:ck}
    H(f) - H(f_\infty^\beta)
    \geq
    C \|f - f_\infty^\beta\|_1^2
  \end{equation}
  for some $C > 0$ which depends only on $\|f_\infty^\beta\|_1$ and
  $\|f_{\infty}^{\beta}\|_{2}$.
\end{lem}

\begin{proof}
  We follow a well-known strategy for proving Csiszár-Kullback-type
  inequalities; see for example \cite{CCD02}.
  For each $v \in \R^d$ call
  \begin{equation*}
    \phi(r) \equiv \phi(r,v) := r \log r - (r+1)\log(r+1) +
    \frac{|v|^2}{2}r,
    \qquad r > 0,
  \end{equation*}
  so that we have
  \begin{equation*}
    H(f) = \ird \phi(f(v))\d v.
  \end{equation*}
  Using that $\phi'(f_\infty^\beta(v))$ is a constant independent of
  $v \in \R^d$ we can write
  \begin{multline}
    \label{eq:1}
    H(f) - H(f_\infty)
    \\
    = \ird \Big(\phi(f(v)) -
    \phi(f_\infty^\beta(v)) - \phi'(f_\infty^\beta(v))(f(v) - f_\infty^\beta(v)) \Big)
    \d v.
  \end{multline}
  Now we notice that $\phi$ is a convex function with
  $\phi''(r) = (r (1+r))^{-1}$, so we have the bound
  \begin{equation*}
    \phi(r) - \phi(s) - \phi'(s)(r - s) \geq \frac{(r-s)^2}{s(1+s)}
    \geq 0
  \end{equation*}
  whenever $0 < r \leq s$. Using this in \eqref{eq:1} and denoting
  $\Omega := \{v \in \R^d \mid f(v) \leq f_\infty^\beta(v)\}$ we
  obtain (writing for convenience $f \equiv f(v)$ and
  $f_\infty^\beta \equiv f_\infty^\beta(v)$ inside the integral)
  \begin{multline*}
    H(f) - H(f_\infty)
    \geq
    \int_\Omega \Big(\phi(f) -
    \phi(f_\infty^\beta) - \phi'(f_\infty^\beta)(f - f_\infty^\beta )\Big)
    \d v
    \\
    \geq
    \int_\Omega \frac{(f - f_\infty^\beta)^2}{f_\infty^\beta (1
      + f_\infty^\beta)} \d v
    \geq
    \left( \int_\Omega (f_\infty^\beta - f ) \d v \right)^2
    \left( \ird f_\infty^\beta
      (v) (1 + f_\infty^\beta) \d v \right)^{-2},
  \end{multline*}
  where for the last step we used the Cauchy-Schwarz inequality.
  Calling
  $$C := \frac{1}{4} \left( \ird f_\infty^\beta (1 + f_\infty^\beta) \d
  v\right)^{-2}$$
  and observing that
  \begin{equation*}
    \int_\Omega (f_\infty^\beta - f) \d v = \frac{1}{2} \|f_\infty^\beta - f\|_1
  \end{equation*}
  this gives precisely
  \begin{equation*}
    H(f) - H(f_\infty)
    \geq
    C \|f - f_\infty^\beta\|_1^2.
  \end{equation*}
\end{proof}

\subsection*{Acknowledgements}

The authors would like to warmly thank the anonymous referee for
useful comments. The authors acknowledge support from the
Spanish project MTM2011-27739-C04-02. J.~A. Ca\~nizo acknowledges
support from the Marie-Curie CIG project KineticCF and the Spanish
project MTM2014-52056-P. J.~A. Carrillo acknowledges support from the
Royal Society through a Wolfson Research Merit Award and the
Engineering and Physical Sciences Research Council (UK) grant number
EP/K008404/1.

\bibliographystyle{abbrv}
\bibliography{bibliography}



\end{document}